\documentclass[11pt]{amsart}
\pdfoutput=1

\usepackage{amscd,latexsym,amsthm,amsfonts,amssymb,amsmath,amsxtra,}
\usepackage[mathscr]{eucal}
\renewcommand\theequation{\thesection.\arabic{equation}}

\usepackage{cases}

\newcommand{\Mscr}{\mathscr{M}}

\newcommand{\ifm}{\mathrm{if}}

\newcommand{\PGL}{\mathrm{PGL}}
\newcommand{\GL}{\mathrm{GL}}

\newcommand{\Trm}{\mathrm{T}}

\newcommand{\omegabar}{\overline{\omega}}

\newcommand{\Cscr}{\mathscr{C}}

\newcommand{\Bscr}{\mathscr{B}}

\newcommand{\Krm}{\mathrm{\mathbf{K}}}

\newcommand{\Erm}{\mathrm{E}}

\newcommand{\F}{\mathrm{\mathbf{F}}}

\newcommand{\crm}{\mathrm{c}}

\newcommand{\Lscr}{\mathscr{L}}

\newcommand{\Nscr}{\mathscr{N}}

\newcommand{\BA}{{\mathbb {A}}}

\newcommand{\BC}{{\mathbb {C}}}

\newcommand{\BN}{{\mathbb {N}}}

\newcommand{\BQ}{{\mathbb {Q}}}
\newcommand{\BR}{{\mathbb {R}}}

\newcommand{\CA}{{\mathcal {A}}}
\newcommand{\CB}{{\mathcal {B}}}
\newcommand{\CC}{{\mathcal {C}}}

\newcommand{\CG}{{\mathcal {G}}}

\newcommand{\CL}{{\mathcal {L}}}

\newcommand{\CO}{{\mathcal {O}}}
\newcommand{\CP}{{\mathcal {P}}}

\newcommand{\CW}{{\mathcal {W}}}

\newcommand{\Fa}{{\mathfrak {a}}}
\newcommand{\Fb}{{\mathfrak {b}}}

\newcommand{\Fl}{{\mathfrak {l}}}
\newcommand{\Fm}{{\mathfrak {m}}}
\newcommand{\Fn}{{\mathfrak {n}}}

\newcommand{\Fp}{{\mathfrak {p}}}
\newcommand{\Fq}{{\mathfrak {q}}}

\newcommand{\Fu}{{\mathfrak {u}}}
\newcommand{\Fv}{{\mathfrak {v}}}

\newcommand{\ScM}{{\mathscr {M}}}

\renewcommand{\Re}{{\mathrm{Re}}}

\newcommand{\bs}{\backslash}

\newtheorem{thm}{Theorem}[section]

\newtheorem{prop}[thm]{Proposition}

\newtheorem {ques/conj}[thm]{Question/Conjecture}

\newtheorem{rmk}[thm]{Remark}

\makeatletter

\newcommand{\Rmnum}[1]{\expandafter\@slowromancap\romannumeral #1@}
\makeatother

\begin{document}
\renewcommand{\theequation}{\arabic{equation}}
\numberwithin{equation}{section}

\title{Spectral Reciprocity for the first moment of triple product $L$-functions and applications}

\author{Xinchen Miao}
\address{Mathematisches Institut\\ Endenicher Allee 60, Bonn, 53115, Germany}
\address{Max Planck Institute for Mathematics\\ Vivatsgasse 7, Bonn, 53111, Germany}
\email{miao@math.uni-bonn.de, olivermiaoxinchen@gmail.com, miao@mpim-bonn.mpg.de}

\date{December, 2024}


\thanks{The author was supported by ERC Advanced Grant  101054336 and Germany's Excellence Strategy grant EXC-2047/1 - 390685813.}

\begin{abstract}

Let $F$ be a number field with adele ring $\BA_F$, $\pi_1, \pi_2$ be two fixed unitary automorphic representations of $\PGL_2(\BA_F)$ with finite coprime analytic conductor $\Fu$ and $\Fv$, $\Fq,\Fl$ be two coprime integral ideals with $(\Fq \Fl, \Fu\Fv)=1$. Following \cite{raphael2}, we estimate the first moment of $L(\frac{1}{2}, \pi \otimes \pi_1 \otimes \pi_2)$ twisted by the Hecke eigenvalues $\lambda_{\pi}(\Fl)$, where $\pi$ runs over unitary automorphic representations of finite conductor dividing $\Fu\Fv\Fq$. By applying the triple product integrals, spectral decomposition and Plancherel formula, we get a reciprocity formula links the twisted first moment of triple product $L$-functions 
to the spectral expansion of certain triple product periods over automorphic representations of finite conductor dividing $\Fl$. As application, we study the subconvexity problem for the triple product $L$-function in the level aspect and give a subconvex bound for $L(\frac{1}{2}, \pi \otimes \pi_1 \otimes \pi_2)$ in terms of the norm of $\Fq$.

\end{abstract}

\maketitle

\tableofcontents

\section{Introduction, Background and History} \label{intro}

Subconvexity estimates belong to the core topics in the theory of analytic number theory and $L$-functions. Moreover, they are one of the most challenging testing grounds for the strength of existing technology.
Let $F$ be a number field with adele ring $\BA_F$, and let $\Pi$ be an automorphic representation of a reductive group $G$. Let $L(s, \Pi)$ be the corresponding $L$-function associated to the representation $\Pi$. If $C(\Pi)$ denotes the analytic conductor of $L(s,\Pi)$, then the famous Phragmen-Lindelof principle gives the upper bound $C(\Pi)^{\frac{1}{4}+\epsilon}$ on the critical line $\Re(s)=\frac{1}{2}$. The subconvexity problem for $L(\frac{1}{2}, \Pi)$ is to establish a non-trivial upper bound of the shape as follows:
$$L(\frac{1}{2}, \Pi) \ll_{F,\epsilon} C(\Pi)^{\frac{1}{4}-\delta+\epsilon},$$
where $\delta$ is some positive absolute constant satisfying $0< \delta \leqslant \frac{1}{4}$ which is independent on $C(\Pi)$.

The subconvexity problem has a very long history for more than one hundred years (See \cite{mic} for more details). We are far from well-understood (except for some lower rank groups). Over the last one hundred years, mathematicians have developed many different approaches to understand the subconvexity problem, for example the classical approach involving the approximate functional equation, the circle (delta) method, Voronoi and Poisson summation formula, also the moment and integral representation approach involving spectral theory on higher rank groups, Petersson and Kuznetsov trace formula, relative trace formulae and so on.

In this paper, we focus on the subconvexity problem of the triple product $L$-function in the finite level aspect, which is the case $G=\GL_2 \times \GL_2 \times \GL_2$. Following \cite{raphael2}, we will use the period integral and moment method approach to establish a spectral reciprocity formula for the twisted first moment of the triple product $L$-function. Finally, via the amplification method, we can break the barrier of the convexity bound and get the subconvexity bound. 

In \cite{miao}, we also consider the subconvexity problem of the triple product $L$-function in the level aspect. However, there exist some differences between these two papers. In \cite{miao}, we vary all the three automorphic representations $\pi_1, \pi_2, \pi_3$ and allow joint ramification and conductor dropping range. Hence, we get the explicit hybrid subconvexity bound in a more general case. In this paper, we fix two cuspidal automorphic representations $\pi_1,\pi_2$ and only vary the automorphic representation $\pi_3$ with finite conductor dividing $\Fq$. The final subconvexity bound is only in terms of $\Fq$, however, the bound we get here is much stronger than \cite{miao} and attaches the limit of the amplification method.

In order to state our results in a more precise way, we need to give some definitions of notations.

Let $\pi_1, \pi_2$ be two fixed unitary cuspidal automorphic representations with fixed finite coprime conductor $\Fu$ and $\Fv$ (i.e.$(\Fu,\Fv)=1$), $\pi_3$ be a unitary automorphic representation of $\PGL_2(\BA_F)$ with finite conductor $\Fq$. Here $\Fq,\Fu,\Fv$ are three integral ideals of $\CO_F$, where $\CO_F$ is the ring of integers of the fixed number field $F$. We further assume that $\Fq$ is coprime with $\Fu\Fv$, i.e. $(\Fq, \Fu \Fv)=1$. The norm of integral ideals $\Fu,\Fv,\Fq$ are $u,v,q$. Therefore, we have $(q,uv)=1$. In order to consider the level aspect subconvexity problem in terms of $q$, we will let $q \rightarrow +\infty$. Hence, without loss of generality, since $u$ and $v$ are positive absolutely bounded integers, we can further assume that $u \ll_{\varepsilon} q^\varepsilon$ and $v \ll_{\varepsilon} q^\varepsilon$ for any $\varepsilon>0$.

We let the real number $\theta$ be the best exponent toward the Ramanujan-Petersson Conjecture for $\GL(2)$ over the number field $F$, we have $0 \leq \theta \leq \frac{7}{64}.$

Let $\Fl$ be an integral ideal of norm $\ell$. We assume that $(\Fl, \Fu\Fv\Fq)=1$, hence $(l, uvq)=1$. We define the following for the cuspidal contribution:

\begin{equation}\label{CuspidalPart}
\mathscr{C}(\pi_1,\pi_2,\Fq,\Fl):= C_1 \cdot \sum_{\substack{\pi \ \mathrm{cuspidal} \\ C(\pi) \vert \Fu \Fv \Fq }}\lambda_\pi(\Fl)\frac{L(\frac{1}{2}, \pi\otimes\pi_1\otimes\pi_2)}{\Lambda(1, \pi,\mathrm{Ad})}f(\pi_\infty)H(\pi,\Fq).
\end{equation}
Here $H$ is certain weight function in terms of finite many ramified non-archimedean local places defined in Section \ref{Connection} and $f(\pi_\infty)$ is defined in Section \ref{SectionInterlude}. Here the constant $C_1$ is a positive constant depending only on the number field $F$ and the nature of the three representations $\pi, \pi_1, \pi_2$. If $\pi,\pi_1,\pi_2$ are all cuspidal, then $C_1=2 \Lambda_F(2)$.

For the continuous part, we denote by $\pi_{\omega}(it)$ the principal series $\omega|\cdot|^{it}\boxplus \omegabar|\cdot|^{-it}$ and define similarly
\begin{equation}\label{ContinuousPart}
\begin{split}
\mathscr{E}(\pi_1,\pi_2,\Fq,\Fl):= C_2 \cdot \sum_{\substack{\omega\in\widehat{\F^\times\setminus \BA_\F^{1}} \\ C(\omega)^2 \vert \Fu \Fv \Fq}}&\int_{-\infty}^\infty \lambda_{\pi_\omega(it)}(\Fl)f(\pi_{\omega_\infty}(it))H(\pi_\omega(it),\Fq) \\ \times & \frac{L(\tfrac{1}{2}+it, \pi_1\otimes\pi_2\otimes\omega)L(\tfrac{1}{2}-it, \pi_1\otimes\pi_2\otimes\omegabar)}{\Lambda^*(1, \pi_\omega(it),\mathrm{Ad})} \frac{dt}{4\pi}.
\end{split}
\end{equation}
In this case, $C_2= 2 \Lambda_F^{*}(1)$ and $H$ is certain weight function in terms of finite many ramified non-archimedean local places defined in Section \ref{Connection}. We also note that the completed $L$-functions satisfy $\Lambda(s,\pi,\mathrm{Ad})=\Lambda(s,\chi^2)\Lambda(s,\chi^{-2})\zeta_F(s)$, where $\pi$ is an Eisenstein series normalized induced from a character $\chi$. In above Equation \ref{ContinuousPart}, $\chi=\omega \vert \cdot \rvert^{it}$ and $\omega$ is a unitary Hecke character. For $\chi^2 \neq 1$, we define $\Lambda^{*}(1,\pi,\mathrm{Ad})=\Lambda(1,\chi^2)\Lambda(1,\chi^{-2})\zeta_F^{*}(1)$, where $\zeta_F^{*}(1)$ is the residue of the Dedekind zeta function at $s=1$, and is a positive real number by the class number formula. We also note that the Dedekind zeta function has a simple pole at $s=1$. When $\chi^2=1$, we define $1/ \Lambda^{*}(1,\pi,\mathrm{Ad}):=0$ (Section 3.2 in \cite{BJN}). Hence, the function $1/ \Lambda^{*}(1,\pi,\mathrm{Ad})$ is continuous in terms of the induced character $\chi$. Since the Dedekind zeta function has a simple pole at $s=1$, if $\chi^2=1$ and $\pi$ is normalized induced from $\chi$, for some small real number $t$ satisfying $\vert t \rvert \leq 1$ (can take zero), we have $1/ \Lambda^{*}(1+it,\pi,\mathrm{Ad}) \gg_F t^2$.

We define
\begin{equation}\label{DefinitionMoment1}
\mathscr{M}(\pi_1,\pi_2,\Fq,\Fl) =\mathscr{C}(\pi_1,\pi_2,\Fq,\Fl)+ \mathscr{E}(\pi_1,\pi_2,\Fq,\Fl).
\end{equation}
The first theorem establishes an upper bound for this twisted first moment.

\begin{thm} \label{moment}
Let $\pi_1,\pi_2$ be two fixed unitary $\theta_i$-tempered ($i=1,2$) cuspidal automorphic representations with finite coprime conductor $\mathfrak {u}$ and $\mathfrak {v}$. We let the real number $\theta_i$ be the best exponent toward the Ramanujan-Petersson Conjecture for $\GL(2)$ over the number field $F$ for $\pi_1$ and $\pi_2$, we have $0 \leq \theta_i \leq \frac{7}{64}$. 
Let $\Fq, \Fl$ be two coprime ideals of $\CO_F$ with the condition $(\Fq \Fl, \mathfrak {u} \mathfrak {v})=1$ and write $q$ and $\ell$ for their respective norms. Therefore, the twisted first moment satisfies
\begin{equation} \label{moment1}
\ScM(\pi_1, \pi_2, \Fq, \Fl) \ll_{\pi_1, \pi_2, F, \varepsilon} (q \ell)^{\epsilon} \cdot \left(\ell^{\frac{1}{2}} \cdot q^{-\frac{1}{2}+\theta}+ \ell^{-\frac{1}{2}+\theta_1+\theta_2} \right).
\end{equation}

\end{thm}

Combining Theorem \ref{moment} with the amplification method, we obtain the following two subconvexity bounds in the level (also depth) aspect.

\begin{thm} \label{subconvex1}
Let $F$ be a number field with ring of integers $\CO_F$. Let $\Fq$ be an integral ideal of $\CO_F$ of norm $q$ and $\pi_3$ a cuspidal automorphic representation of $\PGL_2(\BA_F)$ with finite conductor $\Fq$. Let $\pi_1,\pi_2$ be fixed unitary $\theta_i$-tempered ($i=1,2$) cuspidal automorphic representations with finite fixed coprime conductor $\mathfrak {u}$ and $\mathfrak {v}$. Assume that for all archimedean places $v \vert \infty $, either $\pi_{1,v}$ or $\pi_{2,v}$ is a principal series representation. If $(\Fq, \mathfrak {u} \mathfrak {v})=1$, then for any $\varepsilon>0$, we have the following subconvex estimation
\begin{equation}\label{SubConv1}
L\left( \tfrac{1}{2}, \pi_1 \otimes\pi_2 \otimes\pi_3\right) \ll_{\varepsilon, F,\pi_1,\pi_2,\pi_\infty} q^{1-(\frac{1}{2}-\theta)(1-2\theta_1-2\theta_2)/(3-2\theta_1-2\theta_2)+\varepsilon}.
\end{equation}
Here we have polynomial dependence on $u$ and $v$ in the above subconvexity bound. If we pick $\theta=\theta_1=\theta_2=\frac{7}{64}$, then we have $(\frac{1}{2}-\theta)(1-2\theta_1-2\theta_2)/(3-2\theta_1-2\theta_2)=\frac{225}{2624}>\frac{1}{11.7}$. Hence, we have
$$ L\left( \tfrac{1}{2}, \pi_1 \otimes\pi_2\otimes\pi_3\right) \ll_{\varepsilon, F,\pi_1,\pi_2,\pi_\infty} q^{1-\frac{225}{2624}+\varepsilon}$$
unconditionally. If we further assume the Ramanujan-Petersson Conjecture, we will have
$$ L\left( \tfrac{1}{2}, \pi_1 \otimes\pi_2\otimes\pi_3\right) \ll_{\varepsilon, F,\pi_1,\pi_2,\pi_\infty} q^{1-\frac{1}{6}+\varepsilon}.$$
\end{thm}

As a parallel of Theorem \ref{subconvex1}, we have the following
\begin{thm} \label{subconvex2}
Let $F$ be a number field with ring of integers $\CO_F$. Let $\Fq$ be an integral ideal of $\CO_F$ of norm $q$ and $\chi$ a unitary Hecke character with finite conductor $\Fq$. Let $\pi_1,\pi_2$ be fixed unitary $\theta_i$-tempered ($i=1,2$) cuspidal automorphic representations with finite fixed coprime conductor $\mathfrak {u}$ and $\mathfrak {v}$.
If $(\Fq, \mathfrak {u} \mathfrak {v})=1$, then for any $\varepsilon>0$, we have the following subconvex estimation
\begin{equation}\label{SubConv2}
 L\left( \tfrac{1}{2}, \pi_1\otimes\pi_2 \otimes \chi \right)  \ll_{\varepsilon,\F,\pi_1,\pi_2,\pi_\infty} q^{1-(\frac{1}{2}-\theta)(1-2\theta_1-2\theta_2)/(3-2\theta_1-2\theta_2)+\varepsilon}.
\end{equation}
If we pick $\theta=\theta_1=\theta_2=\frac{7}{64}$, then we have
$$ L\left( \tfrac{1}{2}, \pi_1 \otimes\pi_2\otimes\chi \right) \ll_{\varepsilon, F,\pi_1,\pi_2,\pi_\infty} q^{1-\frac{225}{2624}+\varepsilon}$$
unconditionally.  
\end{thm}

\begin{rmk}
If we further assume that both $\pi_1$ and $\pi_2$ are holomorphic cusp forms which satisfy the Ramanujan-Petersson Conjecture and $\theta_1=\theta_2=0$, then by Theorem \ref{subconvex2}, we have the following inequality if we pick $\theta=\frac{7}{64}$:
$$ L\left( \tfrac{1}{2}, \pi_1 \otimes\pi_2\otimes\chi \right) \ll_{\varepsilon, F,\pi_1,\pi_2,\pi_\infty} q^{1-\frac{25}{192}+\varepsilon}.$$
If we pick $\theta_1=0$ and $\theta_2=\theta=\frac{7}{64}$, we have
$$ L\left( \tfrac{1}{2}, \pi_1 \otimes\pi_2\otimes\chi \right) \ll_{\varepsilon, F,\pi_1,\pi_2,\pi_\infty} q^{1-\frac{625}{5696}+\varepsilon}.$$
Since $\frac{1}{6}> \frac{1}{7.5}> \frac{25}{192}> \frac{1}{7.7}>\frac{1}{9}> \frac{625}{5696}> \frac{1}{9.2}> \frac{1}{11.5}> \frac{225}{2624}>\frac{1}{11.7}> \frac{1}{12}> \frac{1}{16}>\frac{1}{20}> \frac{1}{23}$, then our Theorem \ref{subconvex2} is an improvement of the subconvexity bound for $\GL(1)$ twists of $\GL(2) \times \GL(2)$ Rankin-Selberg $L$-functions both in the finite level and depth aspect (See \cite[Theorem 1]{ghosh} and \cite[Theorem 1]{sun} for more details and comparison).
\end{rmk}

\section{Automorphic Forms Preliminaries}  \label{pre}

In this paper, $F/\mathbb{Q}$ will denote a fixed number field with ring of intergers $\CO_F$ and discriminant $\Delta_F$. We make the assumption that all prime ideals considering in this paper ($\Fq,\Fl,\mathfrak{u},\mathfrak{v}$) do not divide $\Delta_F$. We let $\Lambda_F$ be the complete $\zeta$-function of $F$; it has a simple pole at $s=1$ with residue $\Lambda_F^*(1)$.

For $v$ a place of $F$, we set $F_v$ for the completion of $F$ at the place $v$. We will also write $F_{\Fp}$ if $v$ is finite place that corresponds to a prime ideal $\Fp$ of $\CO_F$. If $v$ is non-Archimedean, we write $\CO_{F_v}$ for the ring of integers in $F_v$ with maximal ideal $\Fm_v$ and uniformizer $\varpi_v$. The size of the residue field is $q_v=\CO_{F_v}/\Fm_v$. For $s\in\BC$, we define the local zeta function $\zeta_{F_v}(s)$ to be $(1-q_v^{-s})^{-1}$ if $v<\infty$, $\zeta_{F_v}(s)=\pi^{-s/2}\Gamma(s/2)$ if $v$ is real and $\zeta_{F_v}(s)=2(2\pi)^{-s}\Gamma(s)$ if $v$ is complex.

The adele ring of $F$ is denoted by $\BA_F$ and its unit group $\BA^\times_F$. We also set $\widehat{\CO}_F:=\prod_{v<\infty} \CO_{F_v}$ for the profinite completion of $\CO_F$ and $\BA^1_F=\{ x\in \BA_F^\times \; : \; \vert x \rvert=1 \}$, where $\vert \cdot \rvert : \BA_F^\times \rightarrow \BR_{>0}$ is the adelic norm map.

We denote by $\psi = \prod_v \psi_v$ the additive character $\psi_{\BQ} \circ \text{Tr}_{F/ \BQ}$ where $\psi_{\BQ}$ is the additive character on $\BQ \setminus \BA_{\BQ}$ with value $e^{2\pi i x}$ on $\BR$. For $v<\infty$, we let $d_v$ be the conductor of $\psi_v$ : this is the smallest non-negative integer such that $\psi_v$ is trivial on $\Fm_v^{d_v}$. In this case, we have $\Delta_F=\prod_{v<\infty} q_v^{d_v}$. We also set $d_v=0$ for the Archimedean local place $v$.

If $R$ is a commutative ring, $\GL_2(R)$ is by definition the group of $2\times 2$ matrices with coefficients in $R$ and determinant in $R^*$. We also defined the following standard subgroups
$$B(R)=\left\{\begin{pmatrix} a & b \\ & d\end{pmatrix} \; : \; a,d \in R^*, b\in R\right\}, \; P(R)= \left\{ \begin{pmatrix} a & b \\ & 1\end{pmatrix} \; : \; a\in R^*, b\in R\right\}, $$
$$Z(R)=\left\{\begin{pmatrix} z & \\ & z\end{pmatrix} \; : \; z\in R^*\right\}, \; A(R)=\left\{\begin{pmatrix} a &  \\ & 1\end{pmatrix} \; : \; a\in R^*\right\}, $$ $$N(R)=\left\{\begin{pmatrix} 1 & b \\ & 1\end{pmatrix} \; : \; b\in R\right\}.$$
We also set
$$n(x)=\begin{pmatrix}
1 & x \\ & 1
\end{pmatrix}, \hspace{0.4cm}w = \begin{pmatrix} & 1 \\ -1 & \end{pmatrix} \hspace{0.4cm} \mathrm{and} \hspace{0.4cm} a(y)= \begin{pmatrix}
y & \\ & 1
\end{pmatrix}.
$$
For any place $v$, we let $K_v$ be the maximal compact subgroup of $G(F_v)$ defined by
\begin{equation}\label{Compact}
K_v= \left\{ \begin{array}{lcl}
\GL_2(\CO_{F_v}) & \text{if} & v \; \mathrm{is \; finite} \\
 & & \\
\mathrm{O}_2(\BR) & \text{if} & v \; \mathrm{is \; real} \\
 & & \\
\mathrm{U}_2(\BC) & \text{if} & v \; \mathrm{is \; complex}.
\end{array}\right.
\end{equation}

We also set $K:= \prod_v K_v$. If $v<\infty$ and $n\geqslant 0$, we define 
$$K_{v,0}(\varpi_v^n):= \left\{ \begin{pmatrix}
a & b \\ c & d
\end{pmatrix} \in K_v \; : \;  c \in \Fm_v^n\right\}.$$
If $\Fb$ is an integral ideal of $\CO_F$ with prime factorization $\Fb=\prod_{v<\infty}\Fp_v^{f_v(\Fb)}$ ($\Fp_v$ is the prime ideal corresponding to the finite place $v$), then we set
$$K_0(\Fb):=\prod_{v<\infty} K_{v,0}\left(\varpi_v^{f_v(\Fb)} \right).$$

We use the same measures normalizations as in \cite{subconvexity}. At each place $v$, $dx_v$ denotes a self-dual measure on $F_v$ with respect to $\psi_v$. If $v<\infty$, $dx_v$ gives the measure $q_v^{-d_v/2}$ to $\CO_{F_v}$. We define $dx=\prod_v dx_v$ on $\BA_F$. We take $d^\times x_v=\zeta_{F_v}(1)\frac{dx_v}{\vert x_v \rvert}$ as the Haar measure on the multiplicative group $F_v^\times$ and $d^\times x = \prod_v d^\times x_v$ as the Haar measure on the idele group $\BA^\times_F$.
We provide $K_v$ with the probability Haar measure $dk_v$. We identify the subgroups $Z(F_v)$, $N(F_v)$ and $A(F_v)$ with respectively $F_v^\times,$ $F_v$ and $F_v^\times$ and equipped them with the measure $d^\times z$, $dx_v$ and $d^\times y_v$. Using the Iwasawa decomposition, namely $\GL_2(F_v)=Z(F_v)N(F_v)A(F_v)K_v$, a Haar measure on $\GL_2(\F_v)$ is given by 
\begin{equation}\label{HaarMeasure}
dg_v = d^\times z dx_v \frac{d^\times y_v}{\vert y_v \rvert}dk_v.
\end{equation} 
The measure on the adelic points of the various subgroups are just the product of the local measures defined above. We also denote by $dg$ the quotient measure on $$X:= Z(\BA_F)\GL_2(F) \setminus \GL_2(\BA_F),$$ 
with total mass $V_F:=\mathrm{vol}(X)<\infty$. 

Let $\pi=\otimes_v\pi_v$ be a unitary automorphic representation of $\PGL_2(\BA_F)$ and fix $\psi$ a character of $\F \setminus \BA_F$. The intertwiner
\begin{equation}\label{NaturalIntertwiner}
\pi \ni \varphi \longmapsto  W_\varphi(g):=\int_{F\setminus \BA_F} \varphi(n(x)g)\psi(-x)dx,
\end{equation}
realizes a $\GL_2(\BA_F)$-equivariant embedding of $\pi$ into a space of functions $W : \GL_2(\BA_F) \rightarrow \BC$ satisfying $W(n(x)g))=\psi(x)W(g)$. The image is called the Whittaker model of $\pi$ with respect to $\psi$ and it is denoted by $\CW(\pi,\psi)$. This space has a factorization $\otimes_v \CW(\pi_v,\psi_v)$ into local Whittaker models of $\pi_v$. A pure tensor $\otimes_v \varphi_v$ has a corresponding decomposition $\prod_v W_{\varphi_v}$ where $W_{\varphi_v}(1)=1$ and is $K_v$-invariant for almost all place $v$.

We define a normalized inner product on the space $\CW(\pi_v,\psi_v)$ by the rule 
\begin{equation}\label{NormalizedInnerProduct}
\vartheta_v(W_v,W_v') :=\zeta_{F_v}(2) \times \frac{\int_{F_v^\times}W_v(a(y))\overline{W}_v'(a(y))d^\times y}{\zeta_{F_v}(1)L(1, \pi_v,\mathrm{Ad})}.
\end{equation}
This normalization has the good property that $\vartheta_v(W_v,W_v)=1$ for $\pi_v$ and $\psi_v$ unramified and $W_v(1)=1$ \cite[Proposition 2.3]{classification}. We also fix for each place $v$ an invariant inner product $\langle \cdot,\cdot\rangle_v$ on $\pi_v$ and an equivariant isometry $\pi_v \rightarrow \CW(\pi_v,\psi_v)$ with respect to \eqref{NormalizedInnerProduct}.

Let $L^2(X)$ be the Hilbert space of square integrable functions $\varphi : X \rightarrow \BC$.  If $\pi$ is a cuspidal representation, for any $\varphi\in\pi$, we can define the $L^2$-norm by
\begin{equation} \label{L^2normCuspidal}
||\varphi||_{L^2}^2:= \int_{X} |\varphi(g)|^2 dg.
\end{equation}
We denote by $L_{\mathrm{cusp}}^2(X)$ the closed subspace of cusp forms, i.e. the functions $\varphi\in L^2(X)$ with the additional property that 
$$\int_{F \setminus \BA_F}\varphi(n(x)g)dg=0, \ \ \mathrm{a.e.} \ g\in \GL_2(\BA_F).$$
Each $\varphi\in L^2_{\mathrm{cusp}}(X)$ admits a Fourier expansion
\begin{equation}\label{FourierSeries}
\varphi(g)= \sum_{\alpha \in F^\times} W_\varphi \left(\begin{pmatrix} \alpha & \\ & 1 \end{pmatrix} g \right),
\end{equation}
\begin{equation}\label{Whittaker-Cuspidal}
W_\varphi(g)=\int_{ F \setminus \BA_F}\varphi\left( \begin{pmatrix} 1 & x \\ & 1 \end{pmatrix} g \right) \psi(-x)dx.
\end{equation}
The group $\GL_2(\BA_F)$ acts by right translations on both spaces $L^2(X)$ and $L_{\mathrm{cusp}}^2(X)$ and the resulting representation is unitary with respect to \eqref{L^2normCuspidal}. It is well known that each irreducible component $\pi$ decomposes into $\pi = \otimes_v \pi_v$ where $\pi_v$ are irreducible and unitary representations of the local groups $\GL_2(F_v)$. The spectral decomposition is established in the first four chapters of \cite{analytic} and gives the orthogonal decomposition
\begin{equation}\label{OrthogonalDecomposition}
L^2(X)=L^2_{\mathrm{cusp}}(X)\oplus L^2_{\mathrm{res}}(X)\oplus L^2_{\mathrm{cont}}(X).
\end{equation}
$L^2_{\mathrm{cusp}}(X)$ decomposes as a direct sum of irreducible $\GL_2(\BA_F)$-representations which are called the cuspidal automorphic representations. $L^2_{\mathrm{res}}(X)$ is the sum of all one dimensional subrepresentations of $L^2(X)$. Finally the continuous part $L^2_{\mathrm{cont}}(X)$ is a direct integral of irreducible $\GL_2(\BA_F)$-representations and it is expressed via the Eisenstein series. In this paper, we call the irreducible components of $L^2_{\mathrm{cusp}}$ and $L^2_{\mathrm{cont}}$ the unitary automorphic representations. If $\pi$ is a unitary representation appearing in the continuous part, we say that $\pi$ is Eisenstein.

For any ideal $\Fb$ of $\CO_F$, we write $L^2(X,\Fb):= L^2(X)^{K_0(\Fb)}$ for the subspace of level $\Fb$ automorphic forms, which is the closed subspace of functions that are invariant under the subgroup $K_0(\Fb)$.

Recall that if $\pi$ is a cuspidal representation, we have a unitary structure on $\pi$ given by \eqref{L^2normCuspidal}. If $\pi$ belongs to the continuous spectrum and $\varphi$ is the Eisenstein series associated to a section $f : \GL_2(\BA_F) \rightarrow \BC$ in an induced representation of $B(\BA_F)$ (see for example \cite[Section 4.1.6]{subconvexity} for the basic facts and notations concerning Eisenstein series), we can define the norm of $\varphi$ by setting 
$$||\varphi||^2_{\mathrm{Eis}}:= \int_{K}|f(k)|^2dk.$$
We define the canonical norm of $\varphi$ by
\begin{equation}\label{CanonicalNorm}
||\varphi||^2_{\mathrm{can}} := \left\{ \begin{array}{lcl}
||\varphi||_{L^2(X)}^2 & \text{if} & \pi \; \mathrm{is \; cuspidal} \\
 & & \\
2\Lambda_\F^*(1) ||\varphi||_{\mathrm{Eis}}^2 & \text{if} & \pi \; \mathrm{is \; Eisenstein},
\end{array}\right.
\end{equation}
Using \cite[Lemma 2.2.3]{subconvexity}, we can compare the global and the local inner product : for $\varphi=\otimes_v \varphi_v \in \pi=\otimes_v\pi_v$ a pure tensor with $\pi$ either cuspidal or Eisenstein and non-singular, i.e. $\pi=\chi_1\boxplus\chi_2$ with $\chi_i$ unitary, $\chi_1\chi_2=1$ and $\chi_1\neq\chi_2,$ we have
\begin{equation}\label{Comparition}
||\varphi||_{\mathrm{can}}^2=2 \Delta_\F^{1/2} \Lambda^*(1,\pi,\mathrm{Ad})\prod_v \langle \varphi_{v},\varphi_v\rangle_v,
\end{equation}
where $\Lambda(s,\pi,\mathrm{Ad})$ is the complete adjoint $L$-function $\prod_v L(s,\pi,\mathrm{Ad})$ and $\Lambda^*(1,\pi,\mathrm{Ad})$ is the first nonvanishing coefficient in the Laurent expansion around $s=1$.
This regularized value satisfies \cite{adjoint}
\begin{equation}\label{BoundAdjoint}
\Lambda^*(1,\pi,\mathrm{Ad})=\mathrm{C}(\pi)^{o(1)}, \; \; \mathrm{as} \; \mathrm{C}(\pi)\rightarrow \infty,
\end{equation}
where $\mathrm{C}(\pi)$ is the analytic conductor of $\pi$, as defined in \cite[Section 1.1]{subconvexity}.

\section{Integral representations of triple product $L$-functions}\label{SectionRankin} 

Let $\pi_1,\pi_2,\pi_3$ be three unitary automorphic representations of $\PGL_2(\BA_F)$ such that at least one of them is cuspidal, say $\pi_2.$ We consider the linear functional on $\pi_1\otimes\pi_2\otimes\pi_3$ defined by
$$I (\varphi_1\otimes\varphi_2\otimes \varphi_3):= \int_X \varphi_1(g)\varphi_2(g)\varphi_3(g)dg.$$
This period is closely related to the central value of the triple product $L$-function $L(\tfrac{1}{2}, \pi_1\otimes\pi_2\otimes\pi_3)$. In order to state the result, we write $\pi_i=\otimes_v \pi_{i,v}$ and for each $v$, we can consider the matrix coefficient
\begin{equation}\label{DefinitionMatrixCoefficient}
I'_v(\varphi_{1,v}\otimes\varphi_{2,v}\otimes\varphi_{3,v}) :=\int_{\PGL_2(\F_v)}\prod_{i=1}^3\langle \pi_{i,v}(g_v)\varphi_{i,v},\varphi_{i,v}\rangle_v dg_v. 
\end{equation}
It is a fact that \cite[(3.27)]{subconvexity} 
\begin{equation}\label{Fact}
\frac{I'(\varphi_{1,v}\otimes\varphi_{2,v}\otimes\varphi_{3,v})}{\prod_{i=1}^3 \langle \varphi_{i,v},\varphi_{i,v}\rangle_v}= \zeta_{F_v}(2)^2 \frac{L(\tfrac{1}{2}, \pi_{1,v}\otimes\pi_{2,v}\otimes\pi_{3,v})}{\prod_{i=1}^3 L(1,\pi_{i,v},\mathrm{Ad})},
\end{equation}
when $v$ is non-Archimedean and all vectors are unramified. It is therefore natural to consider the normalized version
\begin{equation}\label{DefinitionNormalizedMatrixCoefficient}
I_v(\varphi_{1,v}\otimes \varphi_{2,v}\otimes \varphi_{3,v}) := \zeta_{\F_v}(2)^{-2} \frac{\prod_{i=1}^3 L(1,\pi_{i,v},\mathrm{Ad})}{L(\tfrac{1}{2}, \pi_{1,v}\otimes\pi_{2,v}\otimes\pi_{3,v})} I'_v (\varphi_{1,v}\otimes \varphi_{2,v},\varphi_{3,v}).
\end{equation}
The following proposition connects the global trilinear form $I$ with the central value $L(\tfrac{1}{2}, \pi_1\otimes\pi_2\otimes\pi_3)$ and the local matrix coefficients $I_v$. The proof when at least one of the $\pi_i$'s is Eisenstein can be found in \cite[Equation (4.21)]{subconvexity} and is a consequence of the Rankin-Selberg method. The result when all $\pi_i$ are cuspidal is due to Ichino \cite{ichino}.

\begin{prop}\label{PropositionIntegralRepresentation} 
Let $\pi_1,\pi_2,\pi_3$ be unitary automorphic representations of $\PGL_2(\BA_F)$ such that at least one of them is cuspidal. Let $\varphi_i = \otimes_v \varphi_{i,v}\in \otimes_v \pi_{i,v}$ be pure tensors and set $\varphi :=\varphi_1\otimes\varphi_2\otimes\varphi_3$.

\begin{enumerate}
\item If none of the $\pi_i$'s ($i=1,2,3$) is a singular Eisenstein series, then
$$
\frac{|I(\varphi)|^2}{\prod_{i=1}^3 ||\varphi_i||^2_{\mathrm{can}}} = \frac{C}{8\Delta_F^{3/2}}\cdot\frac{\Lambda(\tfrac{1}{2}, \pi_1\otimes\pi_2\otimes\pi_3)}{\prod_{i=1}^3 \Lambda^*(1, \pi_i,\mathrm{Ad})}\prod_v \frac{I_v(\varphi_v)}{\prod_{i=1}^3\langle \varphi_{i,v},\varphi_{i,v}\rangle_v},
$$
with $C=\Lambda_F(2)$ if all $\pi_i$ are cuspidal and $C=1$ if at least one $\pi_i$ is Eisenstein and non-singular.

\item Assume that $\pi_3=1 \boxplus 1$ and let $\varphi_3$ be the Eisenstein associated to the section $f_3(0) \in 1 \boxplus 1$ which for $\Re(s)>0$, is defined as follows:
$$f_3(g,s):= \vert \det(g) \rvert^s \cdot \int_{\BA_F^{\times}} \Phi((0,t)g) \vert t \rvert^{1+2s} d^{\times} t \in \vert \cdot \rvert^s \boxplus \vert \cdot \rvert^{-s},$$
where $\Phi=\prod_v \Phi_v$ and $\Phi_v=1_{\CO_{F_v}}^2$ for finite $v$. Then we have
$$
\frac{|I(\varphi)|^2}{\prod_{i=1}^2 ||\varphi_i||^2_{\mathrm{can}}} = \frac{1}{4 \Delta_F}\cdot\frac{\Lambda(\tfrac{1}{2}, \pi_1\otimes\pi_2\otimes\pi_3)}{\prod_{i=1}^2 \Lambda^*(1, \pi_i,\mathrm{Ad})}\prod_v \frac{I_v(\varphi_v)}{\prod_{i=1}^3\langle \varphi_{i,v},\varphi_{i,v}\rangle_v}.
$$
\end{enumerate}

\end{prop}

\subsection{Hecke operators} \noindent Let $\Fp$ be a prime ideal of $\CO_F$ of norm $p$ and $n\in\mathbb{N}$. Let $F_p$ be the completion of the number field $F$ at the place corresponding to the prime ideal $F_{\Fp}$ and $\varpi_{\Fp}$ be a uniformizer of the ring of integer $\CO_{F_{\Fp}}$. Let $\mathrm{H}_{\Fp^n}$ be the double coset in $\GL_2(F_{\Fp})$ with
$$\mathrm{H}_{\Fp^n}:= \GL_2(\CO_{F_{\Fp}}) \begin{pmatrix}1 &  \\  & \varpi_{\Fp^n}
\end{pmatrix} \GL_2(\CO_{F_{\Fp}}),$$
which, for $n\geqslant 1$, has measure $p^{n-1}(p+1)$ with respect to the Haar measure on $\GL_2(F_\Fp)$ assigning mass $1$ to the maximal open compact subgroup $\GL_2(\CO_{F_\Fp})$ (See \cite[Section 2.8]{sparse}). We consider the compactly supported function:
$$\mu_{\Fp^n}:= \frac{1}{p^{n/2}}\sum_{0\leqslant k\leqslant \frac{n}{2}}\mathbf{1}_{\mathrm{H}_{\Fp^{n-2k}}}.$$
Now for any $f\in \mathscr{C}^\infty(\GL_2(\BA_F))$, the Hecke operator $\Trm_{\Fp^n}$ is given by the convolution of $f$ with $\mu_{\Fp^n}$, i.e. for any $g \in \GL_2(\BA_F)$,
\begin{equation}\label{ActionHecke}
(\Trm_{\Fp^n} f)(g) = (f\star \mu_{\Fp^n})(g):= \int_{\GL_2(F_{\Fp})}f(gh)\mu_{\Fp^n}(h)dh,
\end{equation} 
and the function $h\mapsto f(gh)$ has to be understood under the natural inclusion $\GL_2(F_{\Fp})\hookrightarrow \GL_2(\BA_F)$. This definition extends to an arbitrary integral ideal $\Fa$ by multiplicativity of Hecke operators.


This abstract definition of Hecke operators has a lot of advantages. It simplifies a lot when we deal with $\GL_2(\BA_F)$-invariant functionals. Indeed, consider the natural action of $\GL_2(\BA_F)$ on $\Cscr^\infty(\GL_2(\BA_F))$ by right translation and let $\ell : \Cscr^\infty(\GL_2(\BA_F))\times \Cscr^\infty(\GL_2(\BA_F))\rightarrow \BC$ be a $\GL_2(\BA_F)$-invariant bilinear functional. Then for any $f_1,f_2$ which are right $\GL_2(\CO_{F_{\Fp}})$-invariant, we have the relation
\begin{equation}\label{RelationHecke}
\ell(\Trm_{\Fp^n}f_1,f_2)= \frac{1}{p^{n/2}}\sum_{0\leqslant k\leqslant \frac{n}{2}} \gamma_{n-2k} \cdot \ell\left(\begin{pmatrix} 1 & \\ & \varpi^{n-2k}\end{pmatrix} \cdot f_1,f_2 \right),
\end{equation}
with

\begin{equation}\label{ValueGamma}
\gamma_r:= \left\{ \begin{array}{lcl}
1 & \ifm & r=0 \\ 
 & & \\
p^{r-1}(p+1) & \ifm & r\geqslant 1. 
\end{array}\right.
\end{equation}

\section{Estimations of some period integrals} \label{SectionEstimation}

We recall that in Section \ref{intro}, $\pi_1,\pi_2$ are two fixed unitary $\theta_i$-tempered ($i=1,2$) cuspidal automorphic representations with trivial central character and finite coprime conductor $\mathfrak {u}$ and $\mathfrak {v}$. Let $\varphi_i=\otimes_v \varphi_{i,v}\in \pi_i=\otimes_v \pi_{i,v}$ be pure tensor vectors defined as follows: Since $\pi_1$ and $\pi_2$ are cuspidal, we fix a unitary structure $\langle \cdot,\cdot\rangle_{i,v}$ on each $\pi_{i,v}$ compatible with \eqref{NormalizedInnerProduct} as in previous Section \ref{pre} and take $\varphi_{i,v}$ to be newvectors and have normalized norm $1$. 

Let $\Fl$ be an integral ideal of $\CO_F$ which is coprime to $\Fu,\Fv,\Fq$. From the multiplicativity of the Hecke operators, without loss of generality, we take $\Fl$ of the form $\Fp^n$ with $\Fp\in\mathrm{Spec}(\CO_F)$ and $n\in\mathbb{N}$ and set $p$ for the norm of $\Fp$, so that $\ell=p^n$ is the norm of $\Fl$. For $0\leqslant r\leqslant n$, we write as usual 
$$\varphi_i^{\Fp^r} := \begin{pmatrix} 1 & \\ & \varpi_\Fp^r \end{pmatrix} \cdot \varphi_i.$$ 

\begin{rmk}\label{RemarkInfinite} 
Observe that for every finite place $v$, our local vectors $\varphi_{i,v}$ are uniquely determined since they are normalized newvectors. Indeed if $\pi_i$ ($i=1,2$) is cuspidal, there is a unique L$^2$-normalized new-vector in $\pi_v$. We allow here the infinite component $\varphi_{i,\infty}$ ($i=1,2$) to have a certain degree of freedom. In fact, these will be chosen for Theorem \ref{subconvex1} and \ref{subconvex2} in Section \ref{SectionInterlude} (See Proposition \ref{Hyp}) and will depend on $\pi_{1,\infty}, \pi_{2,\infty}, \pi_{3,\infty}$. Here $\pi_3$ is the automorphic representation which we want to obtain subconvexity bound in Section \ref{SectionSub}. We make therefore the convention that all $\ll$ involved in the following sections depend implicitly on $\varphi_{1,\infty}$ and $\varphi_{2,\infty}$. We can also fix them by choosing for each infinite place $v$, $\varphi_{i,v}\in \pi_{i,v}$ to be the unique normalized vector of minimal weight for $i=1,2$. 
\end{rmk}

\subsection{A $\mathrm{L}^2$-norm}\label{SectionL^2}

We have the following quick upper bound for $\mathrm{L}^2$-norm.

\begin{prop}\label{PropositionL2norm} 
For any real number $\varepsilon>0$, we have the following estimation
$$\int_{X} \left| \varphi_1\varphi_2^{\Fl}\right|^2 \ll_{\pi_1,\pi_2,\varepsilon} \ell^{\varepsilon}.$$
\end{prop}

\begin{proof}
Since $\pi_1$ and $\pi_2$ are both cuspidal, applying Cauchy-Schwarz inequality we get the following bound: $||\varphi_1||_{L^4}^2||\varphi_2||^2_{L^4} \ll_{\pi_1,\pi_2,\varepsilon} \ell^{\varepsilon}$.
\end{proof}

\subsection{A generic term}

Now we gives the following estimation for a particular generic expansion from the spectral decomposition:
\begin{prop}\label{PropositionGeneric} 
For any real number $\varepsilon>0$. Then the generic expansion
\begin{equation}\label{GenericExpansion}
\begin{split}
 & \sum_{\substack{\pi \ \mathrm{cuspidal} \\ \crm(\pi)| \Fl}}\sum_{\psi\in\CB(\pi,\Fl)} \left|\langle \varphi_1\varphi_2^{\Fl},\psi\rangle \right|^2 \\ 
 + & \ \sum_{\substack{\chi\in\widehat{ F^\times\setminus\BA_F^{1}} \\ \crm(\chi)^2|\Fl}}\int_{-\infty}^\infty\sum_{\psi_{it}\in\CB(\chi,\chi^{-1},it,\Fl)}\left|\langle\varphi_1\varphi_2^{\Fl},\Erm(\psi_{it})\rangle \right|^2  \frac{dt}{4\pi} ,
\end{split}
\end{equation}
is bounded, up to a constant depending on $\pi_1,\pi_2, F$ and $\varepsilon$, by $\ell^{\varepsilon}$.
\end{prop}

\begin{proof}
Since both $\pi_i$ ($i=1,2$) are cuspidal,  the expansion \eqref{GenericExpansion} is equal to 
$$\left|\left|\varphi_1\varphi_2^{\Fl}\right|\right|^2_{L^2}-V_{F}^{-1}\sum_{\chi^2=1}\left| \left\langle \varphi_1\varphi_2^\Fl, \varphi_{\chi}\right\rangle\right|^2,$$
where $\varphi_{\chi}(g):=\chi(\det g)$. The $L^2$-norm is bounded by $\ell^{\varepsilon}$ by Proposition \ref{PropositionL2norm}. The one-dimensional contribution (constant term) is zero if $\pi_1$ is not isomorphic to a quadratic twist of $\pi_2$. Otherwise, there exists at most finite many quadratic characters $\chi$ (depending on the number field $F$) such that $\pi_1\simeq \pi_2\otimes \chi$. The cardinality of such quadratic character $\chi$ is $O_{\varepsilon,F}((uv\ell)^{\varepsilon})$. For such a $\chi$, we have using the Hecke relation identity \eqref{RelationHecke} and the $\theta_2$-temperedness:
\begin{equation}\label{BoundIntegral}
\left| \int_{X}\varphi_1\overline{\varphi}_2^{\Fl}\varphi_\chi\right| \leqslant \zeta_{F_{\Fp}}(1)\frac{n+1}{\ell^{1/2-\theta_2}}||\varphi_1||_{L^2}||\varphi_2||_{L^2} \ll_{\varepsilon, \pi_1, \pi_2} \ell^{-1/2+\theta_2+\varepsilon}.
\end{equation}
Since $0 \leqslant \theta_2 \leqslant \frac{7}{64}$, we prove the result.
\end{proof}

\section{A Symmetric Period}\label{SectionSymmetric}
Let $\pi_1,\pi_2$ and $\varphi_i\in\pi_i$ as in Section \ref{SectionEstimation}. We take $\Fq$ which is an integral ideal of $\CO_F$ and $\Fl$ which is an integral ideal of the form $\Fp^n$ with $n\in\BN$ and $\Fp \in \mathrm{Spec}(\CO_F)$ coprime with $\Fq$. Here the integral ideal $\Fq$ is the same as in Section \ref{intro} and is coprime to $\Fu \Fv$. By the multiplicativity of the Hecke operators, without loss of generality, we write $\Fq= \Fq_1^m$ with $m \in \BN$ and $\Fq_1 \in \mathrm{Spec}(\CO_F)$ coprime with $\Fl$ and $\Fp$. We write $q, q_1, p,\ell$ for the norms of $\Fq, \Fq_1, \Fp$ and $\Fl$ respectively. We also recall that the integral ideal $\Fl$ is coprime to $\Fu\Fv\Fq$. We also adopt the convention that all $\ll$ involved in this section depend implicitly on the infinite datas $\varphi_{1,\infty}$ and $\varphi_{2,\infty}$ (See Remark \ref{RemarkInfinite}). By setting
\begin{equation}\label{DefinitionPhi}
\Phi=\varphi_1\varphi_2^{\Fq},
\end{equation}
we also consider the period as in \cite{raphael2}
\begin{equation}\label{ThePeriod}
\CP_\Fq(\Fl,\Phi,\Phi) := \int_{X} \Trm_{\Fl}(\Phi) \overline{\Phi}=\left\langle \Trm_{\Fl}(\Phi),\Phi \right\rangle.
\end{equation}

\subsection{Expansion in the $\Fq$-aspect}
Since $\pi_1$ and $\pi_2$ are cuspidal automorphic representations, We note that $\Phi_1=\varphi_1\varphi_2^{\Fq}$ is a rapid-decay function which is invariant under the congruence group $\Krm_0(\Fu\Fv \Fq)$, we apply Plancherel formula (See \cite[Theorem 2.8]{raphael2}) to the well-defined inner product \eqref{ThePeriod} in the space of forms of level $\Fm\Fn\Fq$. We have the following decomposition of the considered period
\begin{equation}\label{Expansion1}
\CP_{\Fq}(\Fl,\Phi,\Phi)= \CG_{\Fq}(\Fl,\Phi,\Phi)+ \CC_1,
\end{equation}
where the generic part is given by
\begin{equation}\label{GenericExpansionQ}
\begin{split}
 \CG_{\Fq}(\Fl,\Phi,\Phi)= & \sum_{\substack{\pi \ \mathrm{cuspidal} \\ \crm(\pi)| \Fu \Fv \Fq}}\lambda_\pi(\Fl)\sum_{\psi\in\Bscr(\pi,\Fu\Fv\Fq)} \left|\langle \varphi_1\varphi_2^{\Fq},\psi\rangle \right|^2 \\ 
 + & \ \sum_{\substack{\chi\in\widehat{F^\times\setminus\BA_F^{1}} \\ \crm(\chi)^2 \vert \Fu\Fv\Fq}}\int_{-\infty}^\infty \lambda_{\chi,it}(\Fl)\sum_{\psi_{it}\in\Bscr(\chi,\chi^{-1},it,\Fu\Fv\Fq)}\left|\langle\varphi_1\varphi_2^{\Fq},\Erm(\psi_{it})\rangle \right|^2  \frac{dt}{4\pi}.
\end{split}
\end{equation}
Moreover, the constant term $\CC_1$ is the one-dimensional contribution (constant term) which appears only if both $\pi_1$ and $\pi_2$ are cuspidal and there exists finite many quadratic characters $\chi$ (depending on the number field $F$) of $F^{\times} \bs \BA_F^1$ such that $\pi_1\simeq \pi_2\otimes\chi$ ($\chi=1$ if $\pi_1=\pi_2$ for example). The cardinality of such quadratic character $\chi$ is $O_{\varepsilon,F}((uvq)^{\varepsilon})=O_{\varepsilon,F}(q^{\varepsilon})$.
For such a quadratic character $\chi$, by applying the Hecke relation \eqref{RelationHecke}, we have the following:
\begin{equation}
\begin{aligned}
\CC_1= & \, V_{F}^{-1} \cdot \langle \Trm_{\Fl}(\Phi),\varphi_\chi\rangle \langle \varphi_\chi,\Phi\rangle = \frac{\chi(\Fl)\deg(\Trm_{\Fl})}{V_F} \cdot \left|\left\langle \varphi_1\varphi_2^{\Fq},\varphi_\chi\right\rangle\right|^2 \\
=  & \ \frac{\zeta_{\Fq}(2)^2 \chi(\Fl) \deg(\Trm_{\Fl})}{\zeta_{\Fq}(1)^2 V_ F} \cdot \left| \frac{1}{q^{1/2}} \left( \left\langle \varphi_1 \cdot (\Trm_{\Fq}\varphi_2),\varphi_\chi \right\rangle - \frac{1}{q_1} \cdot \left\langle \varphi_1 \cdot (\Trm_{\Fq_1^{m-2}}\varphi_2),\varphi_\chi \right\rangle\right) \right|^2 \\
=  & \ \frac{\zeta_{\Fq}(2)^2 \chi(\Fl) \deg(\Trm_{\Fl})}{\zeta_{\Fq}(1)^2 V_ F} \cdot \left| \frac{1}{q^{1/2}} \left( \lambda_{\pi_2}(\Fq)-\frac{1}{q_1}\lambda_{\pi_2}\left(\Fq_1^{m-2}\right)\right)  \cdot  \left\langle \varphi_1\varphi_2,\varphi_\chi \right\rangle \right|^2 \\
=  & \ \frac{\zeta_{\Fq}(2)^2 \chi(\Fl) \deg(\Trm_{\Fl})}{q \zeta_{\Fq}(1)^2 V_ F} \cdot \left|  \left( \lambda_{\pi_2}(\Fq)-\frac{1}{q_1}\lambda_{\pi_2}\left(\Fq_1^{m-2}\right)\right) \right |^2 \cdot \left| \left\langle \varphi_1\varphi_2,\varphi_\chi \right\rangle \right|^2,
\end{aligned}
\end{equation}
where $\Trm_{\Fq}$ is the Hecke operator and $\zeta_\Fq(s):=\prod_{v|\Fq}\zeta_{F_v}(s)$ is the partial Dedekind zeta function. If the integral ideal $\Fq$ is squarefree, i.e. $m=1$, we define $\Trm_{\Fq_1^{m-2}}:=0$ and $\lambda_{\pi_2}\left(\Fq_1^{m-2}\right):=0$. Hence, $\Trm_{\Fq_1^{m-2}} \varphi_2=0$. Moreover, the degree of the Hecke operator $\Trm_{\Fl}$ is defined by 
\begin{equation}
\deg(\Trm_{\Fl}) := \frac{1}{\ell^{1/2}} \sum_{0\leqslant k\leqslant \frac{n}{2}}\gamma_{n-2k} = \ell^{1/2} \frac{\zeta_{F_\Fp}(1)}{\zeta_{F_\Fp}(n+1)} \leqslant \ell^{1/2}\zeta_{F_\Fp}(1).
\end{equation}
Since $\pi_2$ is $\theta_2$-tempered at all the finite places, we have $\vert \lambda_{\pi_2}(\Fq) \rvert \leqslant (m+1)\cdot q^{\theta_2}$. Applying Cauchy-Schwartz inequality, 
$$\CC_1 \ll_{\pi_1,\pi_2, F,\varepsilon} (\ell q)^\varepsilon \frac{\ell^{1/2}}{q^{1-2\theta_2}} \ll_{\pi_1,\pi_2, F,\varepsilon} (\ell q)^\varepsilon \frac{\ell^{1/2}}{q^{25/32}}$$
since $0 \leqslant \theta_2 \leqslant \frac{7}{64}$.
Hence we can conclude
\begin{equation}\label{FirstRelation}
\CP_{\Fq}(\Fl,\Phi,\Phi)=\CG_{\Fq}(\Fl,\Phi,\Phi)+O_{\pi_1,\pi_2, F,\varepsilon}\left( (\ell q)^\varepsilon \frac{\ell^{1/2}}{q^{1-2\theta_2}}\right),
\end{equation}
where we recall that $\Phi$ is defined as \eqref{DefinitionPhi}.

\subsection{The symmetric relation} The symmetric relation is obtained by grouping differently the vectors $\varphi_i$ ($i=1,2$): In the period $\CP_{\Fq}(\Fl,\Phi,\Phi)$, we first use the Hecke relation \eqref{RelationHecke} to expand the Hecke operator $\Trm_{\Fl}$. Secondly we do the same, but on the reverse way, for the translation by the matrix $\left(\begin{smallmatrix} 1 & \\ & \varpi_{\Fq_1^m} \end{smallmatrix}\right)$. Therefore, this time the Hecke operator $\Trm_{\Fq_1^m}= \Trm_{\Fq}$ appears. We thus obtain the following symmetric relation:
\begin{equation}  \label{Symmetry1}
q^{\frac{1}{2}} \frac{\zeta_\Fq(1)}{\zeta_\Fq(2)}\cdot  \CP_\Fq(\Fl,\Phi,\Phi) = \frac{1}{\ell^{1/2}} \cdot  \sum_{0\leqslant k\leqslant \frac{n}{2}}\gamma_{n-2k}\cdot \left(  \CP_{\Fp^{n-2k}}(\Fq,\Psi_1,\Psi_2)- \frac{1}{q_1} \cdot \CP_{\Fp^{n-2k}}(\Fq_1^{m-2},\Psi_1,\Psi_2)  \right),
\end{equation}
where we simply define $\CP_{\Fp^{n-2k}}(\Fq_1^{m-2},\Psi_1,\Psi_2):=0$ if the integral ideal $\Fq$ is squarefree, i.e. $m=1$. Moreover, 
\begin{equation}\label{Psi}
\Psi_1= \overline{\varphi}_1\varphi_1^{\Fp^{v-2k}} \ \ \mathrm{and} \ \ \Psi_2= \varphi_2\overline{\varphi}_2^{\Fp^{v-2k}}.
\end{equation}
Here Equation \ref{Symmetry1} is a generalization of \cite[Equation 4.9]{raphael2}.

We consider the period $\CP_{\Fp^{n-2k}}(\Fq,\Psi_1,\Psi_2)$ in \eqref{Symmetry1}. The period $\CP_{\Fp^{n-2k}}(\Fq_1^{m-2},\Psi_1,\Psi_2)$ on the right hand side of \eqref{Symmetry1} can be estimated in a similar way and is dominated by the period $\CP_{\Fp^{n-2k}}(\Fq,\Psi_1,\Psi_2)$. We note that the period $\CP_{\Fp^{n-2k}}(\Fq,\Psi_1,\Psi_2)$ admit a similar expansion as \eqref{Expansion1}, but this time over automorphic representations of conductor dividing $\Fp^{n-2k}$. This is the phenomenon of the spectral reciprocity formula. We get a close and interesting relation between different type of first moment of $L$-functions with different spectral length. Hence, we have the following spectral decomposition:
$$\CP_{\Fp^{n-2k}}(\Fq,\Psi_1,\Psi_2)=\CG_{\Fp^{n-2k}}(\Fq,\Psi_1,\Psi_2)+ \mathscr{C}_2(k),$$
where $\CG_{\Fp^{n-2k}}(\Fq,\Psi_1,\Psi_2)$ is the generic part and $\mathscr{C}_2(k)$ is the one-dimensional contribution (constant term).

By definition, we have

\begin{equation}\label{GenericExpansionP}
\begin{split}
  & \CG_{\Fp^{n-2k}}(\Fq,\Psi_1,\Psi_2) := \sum_{\substack{\pi \ \mathrm{cuspidal} \\ \crm(\pi)| \Fp^{n-2k}}}\lambda_\pi(\Fq)\sum_{\psi\in\Bscr(\pi,\Fp^{n-2k})} \langle \Psi_1,\psi\rangle\langle\psi, \Psi_2\rangle \\ 
 + & \; \sum_{\substack{\chi\in\widehat{F^\times\setminus\BA_F^{1}} \\ \crm(\chi)^2|\Fp^{n-2k}}}\int_{-\infty}^\infty \lambda_{\chi,it}(\Fq)\sum_{\psi_{it}\in\Bscr(\chi,\chi^{-1},it,\Fp^{n-2k})}\langle\Psi_1,\Erm(\psi_{it})\rangle \langle\Erm(\psi_{it}),\Psi_2\rangle \frac{dt}{4\pi}.
\end{split}
\end{equation}
Since the automorphic representation $\pi_1$ and $\pi_2$ are cuspidal, the constant term $\mathscr{C}_2(k)$ is bounded by
\begin{equation}\label{ConstantC2}
\begin{aligned}
\vert \mathscr{C}_2(k) \rvert & \leqslant \sum_{\chi} \vert \chi(\Fq) \rvert \frac{\deg(\Trm_{\Fq})\zeta_\Fq(1))}{V_F\zeta_\Fq(2)} \cdot \prod_{i=1}^2\left \vert \int_{X}\varphi_i\overline{\varphi}_i^{\Fp^{v-2k}} \varphi_{i, \chi} \right \rvert \\ 
& \leqslant \sum_{\chi}  V_{F}^{-1}q^{1/2}\frac{\zeta_\Fq(1)^2}{\zeta_\Fq(2)}\prod_{i=1}^2\left \vert \int_{X}\varphi_i\overline{\varphi}_i^{\Fp^{v-2k}} \varphi_{i, \chi} \right \rvert,
\end{aligned}
\end{equation}
where the degree of the Hecke operator $\Trm_{\Fq}$ is defined by 
\begin{equation}
\deg(\Trm_{\Fq}) := \frac{1}{q^{1/2}} \sum_{0\leqslant k\leqslant \frac{m}{2}}\gamma_{m-2k} = q^{1/2} \frac{\zeta_{F_\Fp}(1)}{\zeta_{F_\Fp}(m+1)} \leqslant q^{1/2}\zeta_{F_\Fp}(1),
\end{equation}
since $\zeta_{F_\Fp}(m+1) > 1$.
Moreover, the summation is over quadratic Hecke character $\chi$ satisfying $\pi_1 \cong \pi_1 \otimes \chi$ and $\pi_2 \cong \pi_2 \otimes \chi$ and $\varphi_{\chi}(g)=\chi(\det g)$. We note that the cardinality of such quadratic character $\chi$ is finite (depending on the number field $F$), hence is bounded by $O_{\varepsilon,F}((uv\ell)^{\varepsilon}$.
For such a quadratic character $\chi$, applying Hecke relation identity \eqref{RelationHecke} (See also Equation \ref{BoundIntegral}), for $i=1,2$ we have
$$\left| \int_{X}\varphi_i\overline{\varphi}_i^{\Fp^{n-2k}} \varphi_{i, \chi} \right|\leqslant \zeta_{F_\Fp}(1)\frac{n-2k+1}{p^{\frac{n-2k}{2}(1-2\theta_i)}}||\varphi_i||_{L^2}^2 \Longrightarrow \Cscr_2(k)\ll_{\varepsilon,\pi_1,\pi_2, F}(q\ell)^\varepsilon \frac{q^{1/2}}{p^{(n-2k)(1-\theta_1-\theta_2)}}.$$
Similarly, the constant term $\Cscr_3(k)$ in the period $\CP_{\Fp^{n-2k}}(\Fq_1^{m-2},\Psi_1,\Psi_2)$ can be bounded by $(uvq\ell)^\varepsilon \cdot q_1^{(m-2)/2}/p^{(n-2k)(1-\theta_1-\theta_2)}$. The generic term is almost the same as \eqref{GenericExpansionP} by substituting the integral ideal $\Fq$ to $\Fq_1^{m-2}$.

Now, the total constant term is obtained after summing over $0\leqslant k\leqslant n/2$ as in \eqref{Symmetry1}, i.e.
\begin{equation}\label{GlobalConstantTerm}
\Cscr_2:= \frac{1}{\ell^{1/2}}\sum_{0\leqslant k\leqslant \frac{n}{2}}\gamma_{n-2k} \cdot \left(\Cscr_2(k)-\frac{1}{q_1} \cdot \Cscr_3(k) \right)
\end{equation}
with the following upper bound
\begin{equation}\label{BoundConstantTerm}
\mathscr{C}_2 \ll_{\varepsilon, \pi_1,\pi_2, F} (q \ell)^\varepsilon \cdot \frac{q^{1/2}}{\ell^{1/2-\theta_1-\theta_2}}.
\end{equation}
From the above discussion, we have the following spectral reciprocity relation between the two generic parts:
\begin{equation}\label{ReciprocityRelation}
\begin{aligned}
q^{1/2}\frac{\zeta_\Fq(1)}{\zeta_\Fq(2)} \vert \CG_{\Fq}(\Fl,\Phi,\Phi) \rvert & \leqslant  \frac{1}{\ell ^{1/2}}\sum_{0\leqslant k\leqslant \frac{n}{2}}\gamma_{n-2k} \cdot \vert \CG_{\Fp^{n-2k}}(\Fq,\Psi_1,\Psi_2) \rvert + \Cscr_2 \\
& +\frac{1}{\ell ^{1/2}}\sum_{0\leqslant k\leqslant \frac{n}{2}}\gamma_{n-2k} \cdot \frac{1}{q_1} \cdot \vert \CG_{\Fp^{n-2k}}(\Fq_1^{m-2},\Psi_1,\Psi_2) \rvert + O_{\pi_1,\pi_2, F,\varepsilon}\left( (q\ell)^{\varepsilon} \frac{\ell^{1/2}}{q^{1/2-2\theta_2}} \right).
\end{aligned}
\end{equation}
Now we have to bound the geometric sum $\CG_{\Fp^{n-2k}}(\Fq,\Psi_1,\Psi_2)$. The estimation of the geometric sum $\CG_{\Fp^{n-2k}}(\Fq_1^{m-2},\Psi_1,\Psi_2)$ is almost the same as $\CG_{\Fp^{n-2k}}(\Fq,\Psi_1,\Psi_2)$ and will give a similar bound. Finally, we can estimate the generic terms on the righthand side simply using the bound $\vert \lambda_\pi(\Fq) \rvert 
 \leqslant \tau(\Fq)q^{\theta}$, Cauchy-Schwartz inequality and Proposition \ref{PropositionGeneric}, obtaining 
\begin{equation}\label{FinalBound1}
\frac{1}{\ell^{1/2}} \cdot \sum_{0\leqslant k\leqslant \frac{n}{2}}\gamma_{n-2k} \cdot \CG_{\Fp^{n-2k}}(\Fq,\Psi_1,\Psi_2) \ll_{\pi_1,\pi_2, F,\varepsilon} (\ell q)^\varepsilon \cdot \ell^{1/2} \cdot q^{\theta}.
\end{equation}
Here the real number $\theta$ is the best exponent toward the Ramanujan-Petersson Conjecture for $\GL(2)$ over the number field $F$, we have $0 \leqslant \theta \leqslant  \frac{7}{64}.$
Similarly, we have
\begin{equation}\label{FinalBound2}
\frac{1}{\ell^{1/2}} \cdot \sum_{0\leqslant k\leqslant \frac{n}{2}}\gamma_{n-2k} \cdot \CG_{\Fp^{n-2k}}(\Fq_1^{m-2},\Psi_1,\Psi_2) \ll_{\pi_1,\pi_2, F,\varepsilon} (\ell q)^\varepsilon \cdot \ell^{1/2} \cdot q^{\theta}.
\end{equation}
Since $\frac{1}{2}-2\theta_2>0$, we can rewrite Equation \eqref{ReciprocityRelation} as follows:
\begin{equation} \label{relation2}
q^{1/2}\frac{\zeta_\Fq(1)}{\zeta_\Fq(2)} \vert \CG_{\Fq}(\Fl,\Phi,\Phi) \rvert  \ll_{\varepsilon,F,\pi_1,\pi_2} \Cscr_2+ (\ell q )^{\varepsilon} \ell^{1/2}q^{\theta} \ll_{\varepsilon,F,\pi_1,\pi_2} (\ell q )^{\varepsilon} \cdot \left( \frac{q^{1/2}}{\ell^{1/2-\theta_1-\theta_2}}+ \ell^{1/2}q^{\theta} \right).
\end{equation}

\subsection{Connection with the triple product}\label{Connection}
\noindent We connect in this section the expansion \eqref{GenericExpansionQ} with a first moment of the triple product $L(\tfrac{1}{2}, \pi\otimes\pi_1\otimes\pi_2)$ over automorphic representations $\pi$ of conductor dividing $\Fu\Fv\Fq$. For such a representation $\pi$, we define
\begin{equation}\label{Lscr}
\CL(\pi,\Fq) :=\sum_{\psi \in \CB(\pi,\Fu\Fv\Fq)}\left| \langle \varphi_1\varphi_2^{\Fq},\psi\rangle \right|^2,
\end{equation}
where we recall that $\CB(\pi,\Fu\Fv\Fq)$ is an orthonormal basis of the space of $\Krm_0(\Fu\Fv\Fq)$-vectors in $\pi$. 
By Proposition \ref{PropositionIntegralRepresentation} and Definition \eqref{CanonicalNorm} of the canonical norm, we have
\begin{equation}\label{FactorizationLcal}
\Lscr(\pi,\Fq)=\frac{C}{2\Delta_F^{1/2}}f(\pi_\infty)\frac{L(\tfrac{1}{2}, \pi\otimes\pi_1\otimes\pi_2)}{\Lambda^*(1, \pi,\mathrm{Ad})} \ell(\pi,\Fq),
\end{equation}
where the constant $C=2 \Lambda_F(2)$. If we identify $\pi\simeq \otimes_v\pi_v$, then $\ell(\pi,\Fq)=\prod_{v|\Fu\Fv \Fq}\ell_v$ and the local factors $\ell_v$ are given by the summation of the local triple product integrals in Proposition \ref{PropositionIntegralRepresentation} over an orthonormal basis $\CB(\pi,\Fu\Fv\Fq)$. We define the weight function $$H(\pi,\Fq):= \frac{\ell(\pi,\Fq)}{2\Delta_F^{1/2}}.$$

For the finite place $v|\Fq$, by 
\cite[Theorem 4.1]{hu}, if $\crm(\pi_v)=m \geqslant 1$, we have $\prod_{v \vert \Fq} \ell_v \asymp \frac{1}{q}$. If $0 \leqslant \crm(\pi_v) \leqslant m-1$, by definition, it is known that $\ell_v \geqslant 0$.

If the finite place $v \vert \Fu\Fv$, since $\Fu,\Fv$ are coprime and the corresponding norms $u,v$ are positive absolutely bounded integers, by \cite[Corollary 3.4, Remark 3.4]{BJN}, \cite{hu} and \cite{wood2}, we have $\ell_v \gg 1$. In conclusion, if $C(\pi)=\Fq$, then we have
$$ H(\pi, \Fq) \gg_{\Fu,\Fv} \frac{1}{q}.$$

\subsection{Archimedean function $f(\pi_\infty)$}\label{SectionInterlude}
The Archimedean function $f(\pi_\infty)$ appearing in the factorization \eqref{FactorizationLcal} is given by (See \cite[Equation (3.10)]{raphael}) 
\begin{equation}\label{Definitionfinfty}
f(\pi_\infty):= \sum_{\varphi_\infty \in \CB(\pi_\infty)}I_\infty(\varphi_\infty\otimes\varphi_{1,\infty}\otimes\varphi_{2,\infty})L(\tfrac{1}{2}, \pi_\infty\otimes\pi_{1,\infty}\otimes\pi_{2,\infty}),
\end{equation}
where the local period $I_\infty$ is defined in \eqref{DefinitionNormalizedMatrixCoefficient}. The function $f(\pi_\infty)$ is non-negative and depends on the infinite factors $\pi_{1,\infty}$ and $\pi_{2,\infty}$ and more precisely, on the choice of test vectors $\varphi_{i,\infty}\in \pi_{i,\infty}$ and the orthonormal basis $\CB(\pi_\infty)$. For our application, it will be fundamental that $f$ satisfies the following property: Given $\pi=\pi_\infty\otimes\pi_{\mathrm{fin}}$ a unitary automorphic representation of $\PGL_2(\BA_F)$, there exists $\varphi_{i,\infty}\in\pi_{i,\infty}$, $i=1,2$ with norm $1$ and a basis $\CB(\pi_\infty)$ such that $f(\pi_\infty)$ is bounded below by a power of the archimedean conductor $\crm(\pi_\infty)$. It is a result of Michel and Venkatesh \cite[Proposition 3.6.1]{subconvexity} that such a choice exists when at least one of the local representation $\pi_{1,v}$ and $\pi_{2,v}$ is a principal series when the local place $v \vert \infty$. We give the statement as the following proposition.

\begin{prop}  \cite[Propostion 3.6.1]{subconvexity} \label{Hyp}
Assume that for all archimedean place $v|\infty$, either $\pi_{1,v}$ or $\pi_{2,v}$ is a principal series representation. Then for any $\varepsilon > 0$, there exists a positive constant $C(\pi_{1,\infty},\pi_{2,\infty},\varepsilon)$, such that we have the lower bound
\begin{equation}\label{LowerBound}
f(\pi_\infty) \geqslant \frac{C(\pi_{1,\infty},\pi_{2,\infty},\varepsilon)}{\crm(\pi_\infty)^{1+\varepsilon}}>0.
\end{equation}
\end{prop}

If neither $\pi_{1,v}$ nor $\pi_{2,v}$ is a principal series representation for $v \vert \infty$, and we further assume that $\pi$ is an Eisenstein series, then for every archimedean place $v \vert \infty$, $\pi_v$ is automatically a principal series representation. Without loss of generality, we may assume that the local field $\F_v=\BR$. Since if $\F_v=\BC$, then $\pi_{1,v},\pi_{2,v},\pi_{v}$ must be principal series representations by the classification. Since neither $\pi_{1,v}$ nor $\pi_{2,v}$ is a principal series representation, they must both be discrete series. Applying the non-negativity result in \cite[Corollary 3.4, Remark 3.4]{BJN} and explicit computations in \cite[Proposition 3.4]{wood3}, we can get
$$f(\pi_{\infty}) \geqslant \frac{C(\pi_{\infty},\pi_{2,\infty},\varepsilon)}{e^{(4+\varepsilon) \cdot \crm(\pi_{\infty})}}>0,$$
that is $ f(\pi_{\infty}) \gg_{\pi_{1,\infty},\pi_{2,\infty},\varepsilon} e^{-(4+\varepsilon) \cdot  \crm(\pi_{\infty})} >0,$
which is a weak form of above Proposition \ref{Hyp}, but is enough for our purpose.

\begin{rmk}
In fact, for the local field $\F_v=\BR$, if $\pi_{1,v},\pi_{2,v},\pi_v$ are not all discrete series, i.e. at least one of them is a principal series representation, then we always have
$$ f(\pi_{\infty}) \gg_{\pi_{1,\infty},\pi_{2,\infty},\pi_{\infty}} e^{-(4+\varepsilon) \cdot  \crm(\pi_{\infty})} >0,$$
by the positivity result in \cite[Corollary 3.4, Remark 3.4]{BJN} and explicit computations in \cite[Proposition 3.2, Proposition 3.3, Proposition 3.4]{wood3}, which is enough for our purpose. Hence, in Theorem \ref{subconvex1}, we may slightly weaken the condition that at least one of $\pi_{1,\infty},\pi_{2,\infty}$ and $\pi_{3,\infty}$ is a principal series representation.
\end{rmk}

\begin{rmk}
In Theorem \ref{moment}, \ref{subconvex1} and \ref{subconvex2}, it is possible to remove the condition that two ideals $\Fu$ and $\Fv$ by following the proof idea in \cite[Proposition 3.6.1]{subconvexity} for the non-archimedean local field case using the theory of Kirillov model. The only thing we may need is that the corresponding norms of these two integral ideals are absolutely bounded positive integers.
\end{rmk}

Now the estimation of Theorem \ref{moment} can be achieved from the discussion in Section \ref{SectionEstimation} and \ref{SectionSymmetric} (See also \cite[Section 4.5]{raphael2}).

\section{Proof of Theorem \ref{subconvex1} and Theorem \ref{subconvex2}}  \label{SectionSub}
Let $\Fq$ be an integral ideal of $\CO_F$ and fix $\pi_3$ an automorphic representation (cusp form or Eisenstein series) of $\PGL_2(\BA_F)$ with finite conductor $\Fq$. Let $\pi_1$, $\pi_2$ be two unitary cuspidal automorphic representations satisfying the conditions in Theorem \ref{subconvex1} and Theorem \ref{subconvex2}. We fix the test vectors $\varphi_{i} \in \pi_i$ ($i=1,2$) as in the beginning of Section \ref{SectionEstimation}. 

\subsection{The amplification method}
Let $q^{1/100}<L<q$ be a parameter that we will choose at the end of the proof. We recall that for any $\varepsilon>0$, we have $u \ll_{\varepsilon} q^{\varepsilon}$ and $v \ll_{\varepsilon} q^{\varepsilon}$. Given $\pi$ a unitary automorphic representation of conductor dividing $\Fu\Fv\Fq$, following \cite[Section 12]{blomerspectral} \cite[Section 5.1]{raphael2}, we choose the following amplifier
$$\CA(\pi):= \left(\sum_{\substack{\Fp \in \mathrm{Spec}(\CO_F) \\ \Nscr(\Fp)\leqslant L \\ \Fp \nmid \Fu\Fv\Fq}}\lambda_\pi(\Fp)x(\Fp)\right)^2+\left(\sum_{\substack{\Fp \in \mathrm{Spec}(\CO_F) \\ \Nscr(\Fp)\leqslant L \\ \Fp \nmid \Fu\Fv\Fq }}\lambda_\pi(\Fp^2)x(\Fp^2)\right)^2,$$
where $x(\Fl)=\mathrm{sgn}(\lambda_{\pi_3}(\Fl))$.
By Landau Prime Ideal Theorem and the Hecke relation $\lambda_{\pi_0}(\Fp)^2=\lambda_{\pi_0}(\Fp^2)+1,$ we have
\begin{equation}\label{LowerBound2}
\CA(\pi_3)\geqslant \frac{1}{2}\left(\sum_{\substack{\Fp \in \mathrm{Spec}(\CO_F) \\ \Nscr(\Fp)\leqslant L \\ \Fp \nmid \Fu\Fv\Fq }}|\lambda_{\pi_0}(\Fp)|+|\lambda_{\pi_0}(\Fp^2)|\right)^2\gg_F \frac{L^2}{(\log L)^2}.
\end{equation}
On the other hand, using the Hecke relation again, we have
\begin{equation}\label{DecompositionAmplifier}
\begin{split}
\CA(\pi)=& \; \sum_{\substack{\Fp \in \mathrm{Spec}(\CO_F) \\ \Nscr(\Fp)\leqslant L \\ \Fp \nmid \Fu\Fv\Fq }}(x(\Fp)^2+x(\Fp^2)^2)+\sum_{\substack{\Fp_1,\Fp_2  \\ \Nscr(\Fp_i)\leqslant L \\ \Fp_i \nmid \Fu \Fv \Fq }} x(\Fp_1^2)x(\Fp_2^2)\lambda_\pi(\Fp_1^2\Fp_2^2) \\ + & \; \sum_{\substack{\Fp_1,\Fp_2  \\ \Nscr(\Fp_i)\leqslant L \\ \Fp_i \nmid \Fu \Fv \Fq }} (x(\Fp_1)x(\Fp_2)+\delta_{\Fp_1=\Fp_2}x(\Fp_1^2)x(\Fp_2^2))\lambda_{\pi}(\Fp_1\Fp_2).
\end{split}
\end{equation}
Let $C$, $f(\pi_{0,\infty})$ be the quantity defined respectively in the previous Section \ref{Connection}. If $\pi_3$ is cuspidal, by positivity, we have
$$C \cdot q^{-1} \cdot \CA(\pi_3) \frac{L(\tfrac{1}{2}, \pi_1\otimes\pi_2\otimes\pi_3)}{\Lambda(1, \pi_3,\mathrm{Ad})}f(\pi_{3,\infty})\leqslant \Mscr_\CA(\pi_1,\pi_2,\Fq,\Fl),$$
with $\Mscr_\CA(\pi_1,\pi_2,\Fq,\Fl)$ as in 
\eqref{DefinitionMoment1}, but with the amplifier $\CA(\pi)$ instead of the Hecke eigenvalues in \eqref{CuspidalPart} and \eqref{ContinuousPart}. Using the lower bound \eqref{LowerBound2}, we get 
$$\frac{L\left( \tfrac{1}{2}, \pi_1 \otimes\pi_2\otimes\pi_3 \right)}{\Lambda(1, \pi_3,\mathrm{Ad})}f(\pi_{3,\infty}) \ll_{\varepsilon,\F} L^{-2+\varepsilon} \cdot q \cdot \Mscr_\CA(\pi_1,\pi_2,\Fq,\Fl).$$
Now we expand the amplifier as in \eqref{DecompositionAmplifier} and apply Theorem \ref{moment} with specific integral ideal $\Fl=1,\Fp_1\Fp_2$ or $\Fl=\Fp_1^2\Fp_2^2$ yields the following:
$$
\frac{L\left( \tfrac{1}{2}, \pi_1\otimes\pi_2\otimes\pi_3\right)}{\Lambda(1, \pi_3,\mathrm{Ad})}f(\pi_{3,\infty})  \ll_{\varepsilon, F,\pi_1,\pi_2,\varphi_{1,\infty},\varphi_{2,\infty}} q^\varepsilon\left( q \cdot L^{-1+2\theta_1+2\theta_2}+q^{\frac{1}{2}+\theta} \cdot L^2\right),
$$
Finally, picking $L=q^{(1/2-\theta)/(3-2\theta_1-2\theta_2)}$ (It is easy to see that $q^{1/7}<L \leqslant q^{1/6}$) and we obtain the final subconvexity bound
\begin{equation}\label{HybridBound}
\frac{L\left( \tfrac{1}{2}, \pi_1\otimes\pi_2\otimes\pi_3 \right)}{\Lambda(1, \pi_3,\mathrm{Ad})}f(\pi_{3,\infty})  \ll_{\varepsilon, F,\pi_1,\pi_2,\varphi_{1,\infty}, \varphi_{2,\infty}} q^{1-(\frac{1}{2}-\theta)(1-2\theta_1-2\theta_2)/(3-2\theta_1-2\theta_2)+\varepsilon}.
\end{equation}
Using \eqref{BoundAdjoint} for the adjoint $L$-function at $s=1$ and discussion in Section \ref{SectionInterlude} (Proposition \ref{Hyp}), equation \eqref{HybridBound} transforms into
$$L\left( \tfrac{1}{2}, \pi_1\otimes\pi_2\otimes\pi_3 \right) \ll_{\varepsilon, F,\pi_1,\pi_2,\pi_\infty}q^{1-(\frac{1}{2}-\theta)(1-2\theta_1-2\theta_2)/(3-2\theta_1-2\theta_2)+\varepsilon},$$
which gives the desired subconvexity bound in Theorem \ref{subconvex1}.

If $\pi_3$ is not cuspidal, i.e. an Eisenstein series, the proof of Theorem \ref{subconvex2} is almost the same as above. Instead of the cuspidal distribution and its non-negativity (See Equation \ref{CuspidalPart}), we need the continuous distribution and its positivity (See Equation \ref{ContinuousPart}). We note that in Theorem \ref{subconvex2}, for all archimedean places $v \vert \infty $, the condition that at least one of $\pi_{1,v}$ and $\pi_{2,v}$ is a principal series representation is not assumed. Since in this case $\pi_{3,v}$ itself is a principal series representation, we can get a weak form of Proposition \ref{Hyp} from the discussion in Section \ref{SectionInterlude} and is enough for our purpose on the subconvexity problem. Hence from the discussion between Equation \ref{ContinuousPart} and Equation \ref{DefinitionMoment1}, we see that the only barrier in deducing a subconvex bound is as follows: When the variable $t$ attaches to $0$ and $\chi$ is a quadratic character, the quotient $L(\tfrac{1}{2}+it, \pi_1\otimes\pi_2\otimes\omega)L(\tfrac{1}{2}-it, \pi_1\otimes\pi_2\otimes\omegabar)/{\Lambda^*(1, \pi_\omega(it),\mathrm{Ad})}= \vert L(\tfrac{1}{2}+it, \pi_1\otimes\pi_2\otimes\omega) \rvert^2/ {\Lambda^*(1, \pi_\omega(it),\mathrm{Ad})}$ has a zero of order two at $t=0$. One can overcome this obstacle by an application of Holder's inequality, as in \cite[Section 4]{blo}.

\end{document}